\documentclass[12pt]{amsart}
\usepackage{caption}
\usepackage{pgf,tikz}
\usetikzlibrary{arrows}
\usetikzlibrary{positioning}
\usetikzlibrary{through}
\usetikzlibrary{calc}
\usetikzlibrary{intersections}
\usetikzlibrary{backgrounds}

\usepackage{amsmath,amsthm,amssymb}
\usepackage{mathtools}

\DeclarePairedDelimiter{\abs}{\lvert}{\rvert}

\newcommand{\samewd}[3]{{\newbox\fix\mathchoice{\setbox\fix=\hbox{$\displaystyle #2$}\mathmakebox[\wd\fix][#1]{#3}}{\setbox\fix=\hbox{$#2$}\mathmakebox[\wd\fix][#1]{#3}}{\setbox\fix=\hbox{$\scriptstyle #2$}\mathmakebox[\wd\fix][#1]{#3}}{\setbox\fix=\hbox{$\scriptscriptstyle #2$}\mathmakebox[\wd\fix][#1]{#3}}}}

\newcommand{\Mobius}{M\"obius }

\def\Re{\mathop{\mathsf{Re}}}
\def\Im{\mathop{\mathsf{Im}}}

\newcommand{\Cayley}{\mathcal H}
\newcommand{\iu}{i}

\newcommand{\N}{\mathbb{N}}
\newcommand{\R}{\mathbb{R}}

\newcommand{\C}{\mathbb{C}}

\newcommand{\uH}{\mathbb{H}}
\newcommand{\uHc}{\overline{\mathbb{H}}}
\newcommand{\RS}{\hat{\mathbb{C}}}
\newcommand{\D}{\mathbb{D}}

\newcommand{\UC}{\mathbb{T}}
\newcommand{\diam}{\mathop{\mathsf{diam}}\nolimits}

\newcommand{\de}{\partial}

\newcommand{\clskip}{0.1em}

\newcommand{\clskipii}{0.033em}

\newcommand{\UDc}{\hskip-\clskipii{\overline{\hskip\clskipii\mathbb D\hskip-\clskip}\hskip\clskip}}%
\newcommand{\classSigma}{\mathrm{\Sigma}}

\newcommand{\dfn}{:=}

\newcommand{\compfrac}[2]{\raise.35ex\hbox{$#1$}\kern-.125em\hbox{/}\kern-.1em\lower.4ex\hbox{$#2$}}

\newcommand{\clS}{\mathcal{S}}

\newcommand{\Maponto}
{\xrightarrow{\hbox{\lower.2ex\hbox{$\scriptstyle \smash{\mathsf{onto}}$}}\,}}
\newcommand{\Mapinto}
{\xrightarrow{\hbox{\lower.2ex\hbox{$\scriptstyle \smash{\mathsf{into}}$}}\,}}

\newcommand{\step}[2]{\begin{itemize}\item[\textit{Step\,#1:}]\textit{#2}\end{itemize}}

\newcommand{\STOP}{\par\hbox to\textwidth{\color{red}\leaders\hbox{\,STOP\,}\hfil}\par}

\newcommand{\mcite}[1]{\csname b@#1\endcsname}

\def\dist{\mathop{\mathsf{dist}}}

\def\id{{\sf id}}

\let\Natural=\N
\let\Real=\R
\newcommand{\RealE}{{\hat{\mathbb R}}}
\let\UD=\D
\let\ComplexE=\RS
\let\Complex=\C
\let\UHi=\uH

\newcounter{results}[section]
\newcounter{PreResults}[section]
\theoremstyle{definition}
\newtheorem{definition}[results]{Definition}

\theoremstyle{plain}
\newtheorem{theorem}[results]{Theorem}
\newtheorem{proposition}[results]{Proposition}
\newtheorem{lemma}[results]{Lemma}
\newtheorem{corollary}[results]{Corollary}
\newtheorem{pretheorem}[PreResults]{Theorem}
\newtheorem{preproposition}[PreResults]{Proposition}

\theoremstyle{remark}
\newtheorem{remark}[results]{Remark}
\newtheorem*{unnumremark}{Remark}

\begin{document}

\title[Chordal Loewner Equation]{Chordal Loewner Equation}
\author[A. del Monaco]{Andrea del Monaco}
\author[P. Gumenyuk]{Pavel Gumenyuk}
\address{Dipartimento di Matematica, Universit\`a degli Studi di Roma ``Tor Vergata'',
Via della Ricerca Scientifica 1, 00133, Roma, Italia.}
\email{delmonac@mat.uniroma2.it}\email{gumenyuk@mat.uniroma2.it}
\thanks{Supported by the FIRB programme ``Futuro in Ricerca 2008'', project \textit{Geometria Differenziale Complessa e Dinamica Olomorfa.}}


\begin{abstract}
The aim of this survey paper is to present a complete {\it direct proof} of the well celebrated cornerstone result in Loewner Theory, originally due to Kufarev \textsl{et al}~\cite{KufarevEtAl:FuncUnvHalfPl}, stating that the family of the hydrodynamically normalized conformal self-maps of the upper-half plane onto the complement of a gradually erased slit satisfies, under a suitable parametrization, the chordal Loewner differential equation. The proof is based solely on basic theorems of Geometric Function Theory combined with some elementary topological facts and does not require any advanced technique.
\end{abstract}

\maketitle

\section{Introduction}
In 1923 Loewner~\cite{Loewner:schlichte} introduced a method of the so called \textit{Parametric Representation} for the class~$\clS$ of all univalent holomorphic functions ${f\colon \UD:=\{z\in\Complex\colon |z|<1\}\to\Complex}$ normalized at the origin by $f(0)=0$ and $f'(0)=1$. Much later Kufarev \textsl{et al}~\cite{KufarevEtAl:FuncUnvHalfPl} constructed a similar representation for univalent holomorphic self-maps of the upper half-plane ${\UHi:=\{z\colon \Im z>0\}}$ with the hydrodynamic normalization at~$\infty$. The Parametric Representation Method was further developed by a number of specialists. Without attempting to give an exhaustive bibliography, we only mention the fundamental contributions of Kufarev~\cite{Kufarev:FamiliesOfAnalFunctions} and Pommerenke~\cite{Pommerenke:SubordinationChains},~\cite[Chapter~6]{Pommerenke:UnivalFunctions}. This powerful method has been used a lot in Geometric Function Theory, in particular, as an effective tool to solve extremal problems for univalent functions. One of the most remarkable examples in this connection is the crucial role of the Parametric Representation Method in de Branges' proof~\cite{deBranges:ProofBierberbach} of the famous Bieberbach Conjecture.

In 2000 Schramm in his well-known paper~\cite{Schramm:ScLimUnSpTre} employed the Parametric
Representation Method to study random curves in the plane as it appears to provide
fairly suitable conformally invariant coordinates in the set of all Jordan arcs
in a given simply connected domain joining two prescribed points, one of which
lies on the boundary and the other can be either an interior point
(\textit{radial case}) or a boundary point (\textit{chordal case}).

More details on the history and recent development of the topic can be found in
the survey paper~\cite{ABCDM:EvoLoewDiffEqn}.

One of the fundamental results due to Loewner underlying the Parametric
Representation Method in the radial case can be stated as follows.
\begin{theorem}[see, \textsl{e.g.}, \protect{\cite[Chapter III,
\S2]{Goluzin:GeThFuCo}}]\label{TH_Loewner} Let $D\subset\Complex$ be a simply connected
domain\footnote{The case $D=\Complex$ is \textit{not} excluded.} containing the
origin and let $\gamma$ be a Jordan arc lying in~$D\setminus\{0\}$ except for
one of its end-points, which belongs to~$\partial D$. Then there exists  a
unique continuous function ${\kappa\colon [0,T)\to\UC:=\de \UD}$,
${0<T\le+\infty}$, such that for a suitable parametrization ${\Gamma\colon [0,T]\to\overline D}$ of the arc $\gamma$ with $\Gamma(0)\in D$ and $\Gamma(T)\in\de D$ the family $(f_t)_{t\in[0,T)}$ consisting of the conformal mappings~$f_t\colon \UD\Maponto D\setminus\Gamma([t,T])$ normalized
by~$f_t(0)=0$, $f'_t(0)>0$, satisfies the equation
\begin{equation}\label{EQ_PDE-Loewner-radial}
\frac{\de f_t(z)}{\de t}=z\frac{\de f_t(z)}{\de
z}\frac{\kappa(t)+z}{\kappa(t)-z},
\end{equation}
Moreover, for each $s\in[0,T)$ and each $z\in\UD$ the function
$w_{z,s}(t):=f_t^{-1}\big(f_s(z)\big)$ is the unique solution to the following
Cauchy problem
\begin{equation}\label{EQ_ODE-Loewner-radial}
\frac{dw(t)}{dt}=-w(t)\frac{\kappa(t)+w(t)}{\kappa(t)-w(t)},\quad
t\in[s,T);~\quad w(s)=z.
\end{equation}
\end{theorem}
The functions $f_t$ in the above theorem, mapping a canonical domain (which one
usually chooses to be the unit disk $\UD$ or the upper half-plane~$\UHi$) onto
the complement of a Jordan arc, are colloquially referred to as \textit{(single-)slit
mappings}. The analogue of Theorem~\ref{TH_Loewner} for the chordal case,
\textsl{i.e.} for (single-\!)\,slit mappings of~$\UHi$ into itself normalized at~$\infty$,
is due to Kufarev \textsl{et al}~[\mcite{KufarevEtAl:FuncUnvHalfPl}].
\begin{theorem}[{[\mcite{KufarevEtAl:FuncUnvHalfPl}]}]\label{TH_Kufarev_etal}
Let $\gamma$ be a Jordan arc lying in~$\UHi$ except for one of its end-points,
which belongs to~$\Real$. Then there exists a unique continuous function
${\lambda\colon [0,T]\to\Real}$, ${0<T<+\infty}$, such that for a suitable parametrization
${\Gamma\colon [0,T]\to\UHi\cup\Real}$ of the arc $\gamma$
with $\Gamma(0)\in\UHi$ and $\Gamma(T)\in\Real$, the family
$(g_t)_{t\in[0,T]}$ consisting of the conformal mappings ${g_t\colon
\UHi\Maponto\UHi\setminus\Gamma([t,T])}$ normalized by the expansion
\begin{equation*}\label{EQ_hydrodyn1}
g_t(z)=z+c_1(t)z^{-1}+c_2(t)z^{-2}+\ldots
\end{equation*}
at $z=\infty$ satisfies the equation
\begin{equation}\label{EQ_PDE-Loewner-chordal}
\frac{\de g_t(z)}{\de t}=-\frac{\de g_t}{\de z}\frac{1}{\lambda(z)-z}.
\end{equation}
Moreover, for each $s\in[0,T)$ and each $z\in\UHi$ the function
$w_{z,s}(t):=g_t^{-1}\big(g_s(z)\big)$ is the unique solution to the following
Cauchy problem
\begin{equation}\label{EQ_ODE-Loewner-chordal}
\frac{dw(t)}{dt}=\frac1{\lambda(t)-w(t)},\quad t\in[s,T];~\quad w(s)=z.
\end{equation}
\end{theorem}
Equations~\eqref{EQ_PDE-Loewner-radial} and~\eqref{EQ_ODE-Loewner-radial} are
referred to as the \textit{(radial) Loewner PDE} and \textit{ODE},
respectively, while~\eqref{EQ_PDE-Loewner-chordal}
and~\eqref{EQ_ODE-Loewner-chordal} are known as the \textit{chordal Loewner
differential equations}\footnote{In modern literature the chordal Loewner ODE and PDE contain the extra factor 2 in the right-hand side. Moreover, the chordal Loewner ODE~\eqref{EQ_ODE-Loewner-chordal} is quite often considered with the opposite sign of the right-hand side. See Section~\ref{SS_remarks} at the end of the paper for more details.}.

The proof of Loewner's Theorem~\ref{TH_Loewner} can be found in many text
books, \textsl{e.g.},  in \cite[Chapter~I,
\S2]{Aleksandrov:ParConTheorUnivFunc}, \cite[\S17.3]{Conway:FunComVa2},
\cite[\S3.3]{Duren:UnivFunc}, \cite[Chapter~III, \S2]{Goluzin:GeThFuCo}, and
\cite[Chapter~IX, \S9]{Tsuji:PotenTheoModFuncTheo}. Unfortunately, in all of
these references some important details seem to be missing. A rigorous and
self-contained proof of Theorem~\ref{TH_Loewner} for the case~$D=\Complex$,
based on a number of subtle lemmas related to boundary behaviour of conformal
mappings, can be found in~\cite[Chapter~7]{Hayman:MultivalFunct}.

The chordal case was less well known until the 2000's. A rigorous proof of
Theorem~\ref{TH_Kufarev_etal} can be found in~\cite[\S4.1]{Lawler:ConforInvarProces}. This
proof is based on the deep relationship between the complex Brownian motion and
the harmonic measure, which allows one to use probabilistic methods to study
conformal mappings. Another proof (in a bit more general situation when the
curve $\gamma$ is allowed to ``touch'' itself), which heavily uses techniques
involving the notion of extremal length, can be found in~\cite{LawSchramWend:ValBrowIntExp}.

It is worth to mention that Theorem~\ref{TH_Kufarev_etal} can be deduced as
well from its radial analogue, Theorem~\ref{TH_Loewner}. One of the possible
ways to do so is described in~\cite[Chapter~IV, \S7]{Aleksandrov:ParConTheorUnivFunc}.
Nevertheless, in our opinion, taking into account the increasing interest to
Loewner Theory in general and to the Parametric Representation of slit
mappings, in particular, it is useful to have a detailed elementary direct
proof of this result based solely on Complex Analysis and basic topological
facts. In this survey paper we present such a proof, in a quite self-contained form,
following the idea indicated in the original paper~\cite{KufarevEtAl:FuncUnvHalfPl}.

\section{Preliminaries}
In this section we recall some basic results, which are used in the proof of
Theorem~\ref{TH_Kufarev_etal}. For a set $E\subset\ComplexE$, we will denote by $\overline E$ and $\partial E$ the closure and the boundary of~$E$ w.r.t.~$\ComplexE$, respectively. Moreover, we let
$\RealE:=\overline \Real=\Real\cup\{\infty\}$.

\subsection{Area Theorem} One of the most important elementary results in the theory
of univalent functions was discovered in 1914 by Thomas Hakon Gr\"onwall. Let
$\Delta:=\ComplexE\setminus\UDc$. Denote by $\classSigma$ the class of all
univalent meromorphic functions $g:\Delta\to\ComplexE$ having the Laurent
expansion at~$\infty$ of the form
$$
g(\zeta)=\zeta+b_0+\sum_{n=1}^{+\infty}b_n \zeta^{-n},\quad \zeta\in\Complex\setminus\UDc.
$$
Let $E_g$ stand for the omitted set of~$g$, \textsl{i.e.}
$E_g:=\Complex\setminus g(\Delta)$. Further, given a set $E\subset\C$ we denote by $\diam E$ and $\mathop{\mathsf{area}} E$ its Euclidian diameter and area, respectively, and by $\dist(\cdot,\cdot)$ we denote the Euclidean distance in~$\Complex$.

\begin{pretheorem}[Gr\"onwall's Area Theorem, see, \textsl{e.g.}, \protect{\cite[p.\,29]{Duren:UnivFunc}} or
\protect{\cite[p.\,18]{Pommerenke:UnivalFunctions}}]\label{prop:AreaThm}\mbox{~}\\
Let $g\in\classSigma$. Then
\begin{equation*}
\sum_{n=1}^{+\infty} n|b_n|^2 \le1
\end{equation*}
and the equality holds if and only if~~$~\mathop{\mathsf{area}}E_g=0$.
\end{pretheorem}
As a corollary, one obtains the following two statements.
\begin{preproposition}[\protect{\cite[p.\,19]{Pommerenke:UnivalFunctions}}]\label{prop:EstimClassSig}
Let $g\in\Sigma$. Then $E_g \subset \big\{w\in\C\colon |w-b_0|\le2\big\}$ and
the equality holds if and only if $E_g$ is a line segment of length $4$.
\end{preproposition}

\begin{lemma}\label{LM:varphi_extended}
Let $\varphi\colon\ComplexE\setminus K_1\Maponto \ComplexE\setminus K_2$, where
$K_1,\,K_2\subset\Complex$ are two compact sets, be a conformal mapping with
the Laurent expansion at~$\infty$ of the form
\begin{equation}\label{EQ:expansion_in_lemma}
\varphi(z)=z+\sum_{n=1}^{+\infty} c_nz^{-n}.
\end{equation}
Then the following statements hold:
\begin{itemize}
\item[(i)] $|c_1|\le
\min_{j=1,2}\big(\diam K_j\big)^2$;
\item[(ii)] $K_1\subset\big\{\hphantom{w}\mathllap{z}\colon |z-w_0|\le 2 \diam K_2\big\}$ for any $w_0\in K_2$;
\item[(iii)] $K_2\subset\big\{w\colon |w-z_0|\le 2 \diam K_1\big\}$ for any $z_0\in K_1$;
\item[(iv)] Let $z_1\in\Complex\setminus K_1$ and $z_2:=\varphi(z_1)$. If $\dist(z_j,K_j)>\diam K_j$ for $j=1$ \textbf{or} $j=2$, then $|z_1-z_2|<3\diam K_j$ for the same value of~$j$.
\end{itemize}
\end{lemma}
\begin{proof}
Denote $R_j:=\diam K_j$ for $j=1,2$. Fix any $z_0\in K_1$. Then $K_1\subset
\{z:|z-z_0|\le R_1\}$. Therefore the function
$g(\zeta):=\varphi\big(R_1\zeta+z_0\big)/R_1$, $\zeta\in\Delta$, belongs to the
class~$\Sigma$, with the free term in its Laurent expansion at~$\infty$ equal to~$b_0=z_0/R_1$. Since by construction $K_2\subset\{w\colon w/R_1\in E_g\}$,
it follows from Proposition~\ref{prop:EstimClassSig} that $K_2\subset\{w\colon |w-z_0|\le 2R_1\}$. This proves~(iii).

In order to estimate~$c_1$ consider again the function $g$ and apply the Area
Theorem (Theorem~\ref{prop:AreaThm}), from which it follows then that $|b_1|=|c_1|/R_1^2\le1$.

Further, if $\dist(z_1,K_1)>\diam K_1$, then $\zeta_1:=(z_1-z_0)/R_1\in\Delta$ and therefore $|z_2-z_1|=\big|R_1g(\zeta_1)-(R_1\zeta_1+z_0)\big|=R_1|g(\zeta_1)-b_0-\zeta_1|=R_1|f(\zeta_1)|$, where the function $f$ defined by $f(\zeta):=g(\zeta)-b_0-\zeta$ for all $\zeta\in\Delta\setminus\{\infty\}$ and $f(\infty)=0$ is holomorphic in~$\Delta$. Applying the Maximum Modulus Principle to~$f$ and Proposition~\ref{prop:EstimClassSig} to~$g$, we conclude that
$$
|z_2-z_1|\le R_1\limsup_{\Delta\ni \zeta\to\UC}|f(\zeta)|\le R_1\limsup_{\Delta\ni \zeta\to\UC}\big(|g(\zeta)-b_0|+|\zeta|\big)\le 3R_1.
$$
This proves (iv) for~$j=1.$\vskip1ex

\noindent To complete the proof of the lemma it remains to apply the above arguments for~$\varphi^{-1}$.
\end{proof}

\subsection{Schwarz formula for the upper half-plane}
We will need the following version of the Schwarz Integral Formula.
\begin{proposition}\label{pr_SchwarzIntegralForumula}
Let $f:\uHc\to\ComplexE$ be continuous in~$\uHc=\uH\cup\RealE$
and holomorphic in~$\uH$. Suppose that $f(\infty)=0$ and
\begin{equation}\label{pr_hp_IntAbsConv}
\int_\R\,\abs[\Bigg]{\frac{\Im\big\lbrace f(\xi)\big\rbrace}{\xi-\iu}}\,d\xi
\;<\; +\infty\,.
\end{equation}
Then
\begin{equation}\label{pr_th_SchwIntFor}
f(z) \,=\,
\frac{1}{\pi}\,\int_\R\,\frac{\Im\big\lbrace f(\xi)\big\rbrace}{\xi-z}\,d\xi
\end{equation}
for all $z\in\uH$.
\end{proposition}
\begin{proof} Applying the Schwarz Integral Formula to the function
$\psi(z):=-\iu f\big(\Cayley(z)\big)$, where $\Cayley(z):=\iu(1+z)/(1-z)$ is
the Cayley map of~$\UD$ onto the upper half plane~$\uH$, we get
\begin{align}
-\iu f\big(\Cayley(z)\big)\; &=\; \frac{1}{2\pi}\int_{\UC}\frac{\omega +
z}{\omega - z} \,\,\Re\big\{-\iu f\big(\Cayley(\omega)\big)\big\}
\,\abs{d\omega} + \iu\,\Im\big\{-\iu f\big(\Cayley(0)\big)
\big\}\notag\\
\label{pr_prf_Form1}&=\; \frac{1}{2\pi} \int_{\UC\setminus\{1\}} \frac{\omega +
z}{\omega - z}\,\,\Im\big\{ f\big(\Cayley(\omega)\big)\big\}\,\abs{d\omega} +
\iu{C},
\end{align}
where $C:=-\Re\{ f(\iu)\}$. Substituting $z:=\Cayley^{-1}(w)$, $w\in\uH$,
in~\eqref{pr_prf_Form1}, and changing the integration variable
$\omega:=\Cayley^{-1}(\xi),\;\xi\in\R$, we get
\begin{align*}
 f(w) \;&=\; \frac{1}{\pi} \int_{\R}
\frac{1+\xi{w}}{\xi-w}\,\,\Im\big\lbrace f(\xi)\big\rbrace\,\frac{d\xi}{\xi^2+1}
\;-\; C = I_{1}(w) \;-\; I_{2} \;-\; C
\end{align*}
for all $w\in\uH$, where the integrals
\begin{equation*}
I_{1}(w):=\frac{1}{2\pi}\int_{\R}\frac{\Im\big\{ f(\xi)\big\}}{\xi-w}\,d\xi
\quad\text{and}\quad
I_{2}\dfn\frac{1}{2\pi}\int_{\R}\frac{\xi}{\xi^2+1}\,\,\Im\big\{ f(\xi)\big\}\,d\xi
\end{equation*}
converge absolutely thanks to~\eqref{pr_hp_IntAbsConv}. Moreover,
condition~\eqref{pr_hp_IntAbsConv} implies also  that
$I_{1}(\iu{y})\to0$ as $y\to+\infty$. Thus $-I_{2}-C=f(\infty)=0$ and we get formula~\eqref{pr_th_SchwIntFor}.
\end{proof}

\subsection{Boundary behaviour of slit mappings}
We start with  two basic definitions.
\begin{definition}
Let $D\subsetneq\RS$ be a domain. A subset $\gamma$ of $D$ is called a
\emph{slit in $D$} if there exists a homeomorphism
$\Gamma:[0,T]\Mapinto\overline{D}$, where $T>0$, such that
$\Gamma\big([0,T)\big)=\gamma$ and $\Gamma(T)\in\partial{D}$. The function
$\Gamma$ is said to be a \emph{parameterization} of the slit $\gamma$. The point~$\Gamma(0)$ is called the \emph{tip of the slit}~$\gamma$ and the
point $\Gamma(T)$ is called the \emph{root} (or the \emph{landing point}) of
$\gamma$. We will also say that \emph{$\gamma$ lands at $\Gamma(T)$}.
\end{definition}

\begin{definition}
In what follows by a \emph{single-slit mapping} we will mean\footnote{More
generally, a \emph{single-slit mapping} of a domain $U$ into a domain~$D$ is a
conformal map of $U$ onto $D$ minus a slit. However, in this paper we will be
restricted to the case when $U=D=\uH$ and the slit $D\setminus g(U)$ lands at a
finite point.} a conformal map $g:\uH\Mapinto\uH$ such that $\uH\setminus
g(\uH)$ is a slit in~$\uH$ landing at some point on~$\R$.
\end{definition}

The following theorem implies easily that \emph{single-slit mappings admit
continuous extension to the boundary}.

\begin{pretheorem}[see,\textsl{e.g.}, \protect{\cite[Chapter~9,
Theorem~9.8]{Pommerenke:UnivalFunctions}}]\label{TH_cont_ext} A conformal
mapping $g:\uH\Mapinto\ComplexE$ admits a continuous extension ${\hat
g:\uHc\to\ComplexE}$ if and only if $\partial g(\uH)$ is a locally connected
set.
\end{pretheorem}
\begin{unnumremark}
It is worth to mention that combining~\cite[Chapter~3,
Lemma~1]{ColLoh:ClusterSets} and arguments in~\cite[Chapter~9,
Theorem~9.8]{Pommerenke:UnivalFunctions} with~\cite[Chapter~3,
Lemma~3.29]{HockingYoung:Topology}, one can give a direct elementary proof of
the above theorem avoiding the usage of the theory of prime ends and that
of normal functions (the No-Koebe-Arcs Theorem).
\end{unnumremark}
\begin{remark}\label{RM:SlitLocConn}
Note that if $D$ is $\uH$ minus a slit, then $\partial D=\RealE\cup\gamma$ is locally connected as a union of two closed locally connected sets. Hence, any single-slit mapping $g:\uH\Mapinto\uH$ admits a continuous extension $\hat g\colon \uHc=\uH\cup\RealE\to\RS$.
\end{remark}
Since the extension in Theorem~\ref{TH_cont_ext} and in the above remark is unique, we will denote it  by the same symbol as the conformal map itself omitting the sign ``\,$\hat{~}\,$'', while any other extension which might disagree with the one under consideration will be denoted in a different way.

The following theorem is one of the key points in the proof of Theorem~\ref{TH_Kufarev_etal}.
\begin{pretheorem}\label{Th_PreOfSlitIsSlit}
Let $g \colon \uH \Mapinto \RS$ be a conformal mapping and $\gamma$ a slit in
the domain $D:=g(\uH)$. Then the set $g^{-1}(\gamma)$ is a slit in $\uH$.
\end{pretheorem}

\begin{unnumremark}
The above theorem means that if $\Gamma\colon[0,T]\longrightarrow\overline D$
is any parametrization of $\gamma$, then $g^{-1}\circ\Gamma\big\rvert_{[0,T)}$
has a continuous extension to the point $t=T$.
\end{unnumremark}

\begin{unnumremark}
It might be useful to have a simple proof of Theorem~\ref{Th_PreOfSlitIsSlit}
for the case when $\partial g(\uH)$ is locally connected and hence the
function~$g$ can be extended to a continuous map of~$\uHc$
into~$\RS$. First of all, note that we can pass to a
conformal map of~$\UD$ with a continuous extension to~$\UDc$, which we again
denote by~$g$. Now suppose on the contrary to the statement of
Theorem~\ref{Th_PreOfSlitIsSlit} that $g^{-1}(\gamma)$ is not a slit in $\D$.
Then there exists an arc $\mathcal{C}\subset\UC$ not reducing to a point such
that $g(\mathcal{C})=\lbrace\xi_0\rbrace$, where $\xi_0$ is the root
of~$\gamma$. Using the Schwarz Reflection Principle and the Uniqueness
Principle for holomorphic functions we conclude that $g\equiv\xi_0$, which is
not possible.
\end{unnumremark}

Assume that $g:\uH\Maponto D$ is a conformal map and $\partial D$ is locally
connected. The continuous extension of~$g$ given by Theorem~\ref{TH_cont_ext}
does not need to be injective on the boundary. The following statement allows
to understand better the mapping properties of~$g\rvert_{\partial\uH}$.

\begin{preproposition}\label{Pr_PreimCutPoi}
In the above notation, let $w_0\in\partial D$ and
$\mathcal{W}:=g^{-1}\left(\{w_0\}\right)$. Then the map
$\RealE=\partial \uH\supset\mathcal C~~\mapsto~~g(\mathcal
C)\subset\partial D$ establishes a bijective correspondence between the
connected components of $\RealE\setminus\mathcal{W}$ and those of $\partial
D\setminus\{w_0\}$. In particular, the set~$\mathcal W$ consists of
$\nu\in\Natural$ pairwise distinct points if and only if
$\partial{D}\setminus\{w_0\}$ has exactly $\nu$ connected components.
\end{preproposition}
\noindent The proof of this proposition can be found, \textsl{e.g.},
in~\cite[Chapter~2, Proposition~2.5]{Pommerenke:BoundBehaConfMaps}.

Taking into account that an injective continuous mapping of an interval of the
form $(a,b)$ or $[a,b]$, $a<b$, into a simple curve has always the continuous
inverse, from Proposition~\ref{Pr_PreimCutPoi} one easily obtains the following
statement.
\begin{proposition}\label{Pr_PropertiesSingleSlitMap}
Let $g:\uH\Mapinto\uH$ be a single-slit map with $g(\infty)=\infty$, $\gamma:=\uH\setminus g(\uH)$. Then the following assertions hold:
\begin{enumerate}
\item[(i)] the preimage $g^{-1}(\xi_0)$ of the
root~$\xi_0$ of the slit $\gamma$ consists exactly of two points
$\alpha,\beta\in\R$, $\alpha<\beta$;
\item[(ii)] the preimage $g^{-1}(\omega_0)$ of the tip $\omega_0$
of the slit $\gamma$ consists of a unique point $\lambda\in(\alpha,\beta)$;
\item[(iii)] $g$ maps $\RealE\setminus[\alpha,\beta]=(\beta,+\infty)\cup\{\infty\}\cup(-\infty,\alpha)$ homeomorphically
onto $\RealE\setminus\lbrace{\xi_0}\rbrace$;
\item[(iv)] each of the segments $[\alpha,\lambda]$ and $[\lambda,\beta]$ is mapped by $g$
homeomorphically onto ${\bar\gamma:=\gamma\cup\{\xi_0\}}$.
\end{enumerate}
\end{proposition}

\section{The chordal Loewner equation}

In this section we present a detailed elementary proof of Theorem~\ref{TH_Kufarev_etal}.

\subsection{``Chordal version'' of the Riemann Mapping Theorem}
\begin{proposition}\label{Prop_RiemMapThm}
Let $\gamma$ be a slit in the upper half-plane $\uH$ landing at some point $\xi_0\in\R$ and $\bar\gamma:=\gamma\cup\{\xi_0\}$. Then there exists a unique single-slit mapping $g_\gamma:\uH\Maponto H:=\uH\setminus\gamma$ satisfying the hydrodynamic condition
\begin{equation}\label{eq_HydroDynCond}
\lim_{\mathclap{z\to\infty}}\,\, g(z)-z = 0.
\end{equation}
Moreover, the following statements hold:
\begin{itemize}
\item[(i)]{$g_\gamma\big\rvert_{\uH}$ extends to a conformal map $g_\gamma^{*}\!$ of
$\RS\setminus\mathcal{C}$ onto $\RS\setminus(\bar\gamma\cup\bar\gamma^*)$, where $\mathcal{C}:=g_\gamma^{-1}(\bar\gamma)$ and $\bar\gamma^{*}$ is the reflection of $\bar\gamma$ with respect to $\Real$;}%
\item[(ii)]{$g_\gamma^{*}$ has a Laurent expansion at $\infty$ of the form
\begin{equation}\label{eq_FormLaurSerG}
g_\gamma^{*}(z) = z + \sum_{n=1}^{\infty}c_{n}z^{-n},
\end{equation}
with $c_n\in\R$ for all $n\in\N$ and $c_1<0$.}
\end{itemize}
\end{proposition}
\begin{proof} The proof is divided into 3 steps.\mbox{~}
\step1{we prove first the existence of the map $g_\gamma$, assertion (i) and expansion~\eqref{eq_FormLaurSerG}.}
Since $H\subsetneq\Complex$ is a simply connected domain\footnote{To be completely rigorous, one has to use here some basic topological arguments, including Janiszewski Theorem. See \cite[\S1.5]{Pommerenke:UnivalFunctions}.}, according to the Riemann Mapping Theorem, there exists a conformal map $g_0$ of $\D$ onto $H$. As we already mentioned (see Remark~\ref{RM:SlitLocConn}) $g_0$ extends continuously to $\uHc$. Recall that in such case, we use the same notation for the extended map from $\uHc$
into~$\ComplexE$. Moreover, precomposing, if necessary, $g$ with a \Mobius transforation of~$\uH$, we may assume that $g_0(\infty)=\infty$.

Therefore, by Proposition~\ref{Pr_PropertiesSingleSlitMap}, $g_0\big(\partial\uH\setminus[\alpha,\beta]\big)=\partial\uH\setminus\{\xi_0\}$, where $[\alpha,\beta]=g_0^{-1}(\bar\gamma)$. Hence, by the Schwarz Reflection Principle, $g_0\rvert_{\uH}$ can be extended to a conformal map $g_{0}^{*}$ of $\RS\setminus[\alpha,\beta]$ onto $\RS\setminus(\bar\gamma\cup\bar\gamma^*)$.

Since $g^*_{0}(\infty)=g_0(\infty)=\infty$ and $g^*_{0}$ is a conformal map, we see that $\infty$ is a simple pole of $g_{0}^{*}$. As a consequence, the map $g_{0}^{*}$ has a Laurent expansion at $\infty$ of the form
\begin{equation*}
g_{0}^*(z) = az + b + \sum_{n=1}^{\infty}c_{n}^0z^{-n},
\end{equation*}
where $a\neq0$. Furthermore, note that $\overline{g_{0}^{*}(\bar{z})} = g_{0}^{*}(z)$ for all $z\in\Real\setminus[\alpha,\beta]$ and that both sides in this equality are holomorphic in $z$ on $\C\setminus[\alpha,\beta]$. Therefore, the equality holds for all $z\in\C\setminus[\alpha,\beta]$. It follows that the coefficient $a$, $b$, and $c_n^0$, for all $n\in\N$, are real.\\
Taking into account that
\begin{equation*}
a\,=\,\Re a\,=\,\Re\,\,\lim_{\mathclap{y\to+\infty}}\frac{g_0(iy)}{iy}=\,
\lim_{\mathclap{y\to+\infty}}\frac{\Im g_0(iy)}{y}
\end{equation*}
and that $\Im{g_0(iy)}>0$ for all $y>0$, we finally conclude that $a>0$.

Since $a>0$ and $b\in\Real$, the linear  function $L(z):=az+b$ is a \Mobius transformation of~$\uH$ and hence $g_\gamma:=g_{0}\circ L^{-1}$ is a conformal map of $\uH$ onto $H$.  Furthermore, an easy computation shows that the extension~$g_\gamma^*=g^*_{0}\circ L^{-1}$ of~$g_\gamma|_{\uH}$ to $\ComplexE\setminus\mathcal{C}$, where $\mathcal C:=g_\gamma^{-1}(\bar\gamma)=L([\alpha,\beta])$, is represented in a neighborhood of~$\infty$ by the Laurent expansion~\eqref{eq_FormLaurSerG} with all coefficients~$c_n\in\Real$. This completes Step 1.

\step2{now we show that $c_1<0$.} Since $g_\gamma(z)-z$ is not constant in~$\uH$, applying the Maximum Principle to the harmonic function $\uH\ni z\mapsto \Im(z-g_\gamma(z))$, which extends to a continuous real-valued function on $\uHc$, we conclude that the holomorphic function $h(z):=g_\gamma(z)-z=c_1/z+c_2/z^2+\ldots$ maps $\uH$ into itself. Set
$$
k_0:=\min\{k\in\Natural\colon c_k\neq0\}\quad\text{and}\quad
\theta_0:=\frac\pi{2k_0}\left(2-\frac{c_{k_0}}{|c_{k_0}|}\right).
$$
If $c_1\ge0$, then $\theta_0\in(0,\pi)$. In this case we would have $\Im h<0$ on the ray $z=\rho e^{i\theta_0}\in\uH$ for all $\rho>0$ large enough. Thus $c_1<0$.

\step3{it remains to show that the map~$g_\gamma$ is unique.} Let $\tilde g_\gamma$ be another conformal mapping of~$\uH$ onto~$H$ satisfying~\eqref{eq_HydroDynCond}. Then $L:=\tilde g_\gamma^{-1}\circ g_\gamma$ is a \Mobius transformation of~$\uH$ fixing~$\infty$. Therefore, it is of the form $L(z)=az+b$. Furthermore,
$g_\gamma(z)-z=\tilde g_\gamma(az+b)-z=(a-1)z+b+o(1)$ as $z\to \infty$. Thus, by~\eqref{eq_HydroDynCond}, $a=1$, $b=0$, and consequently $\tilde g_\gamma=g_\gamma$. The proof is now complete.
\end{proof}

\subsection{Standard parametrization of slits in~$\uH$.}\label{SS_standParam}
Throughout this subsection we consider a slit $\gamma$ in~$\uH$ landing at some point~$\xi_0\in\Real$. Let $\Gamma\colon[0,T]\longrightarrow\uHc$, $T>0$, be an arbitrary parametrization of this slit. For each $t\in[0,T)$, the set $\gamma_t:=\Gamma\big([t,T)\big)$ is a slit in~$\uH$. Hence by Proposition~\ref{Prop_RiemMapThm} there exists a unique single-slit map $g_{\gamma_t}$ satisfying the hydrodynamic normalization~\eqref{eq_HydroDynCond} such that $g_{\gamma_t}(\uH)=\uH\setminus\gamma_t$. Denote by~$c_1(t)$ the value of the coefficient $c_1$ in Laurent expansion~\eqref{eq_FormLaurSerG} of $g_{\gamma_t}^*$. To include the case $t=T$ we set $\gamma_T:=\emptyset$, $g_{\gamma_t}:=\id_{\uH}$ and, correspondingly, $c_1(T):=0$.

\begin{definition}
A parametrization  $\Gamma\colon[0,T]\longrightarrow\uHc$, $T>0$, of the slit $\gamma$ is said to be a \emph{standard parametrization of} $\gamma$ if  $c_1(t)=t-T$ for all $t\in[0,T]$.
\end{definition}
The main result of this subsection is as follows.
\begin{proposition}\label{PR_standardParam}
There exists a unique standard parametrization $\Gamma_0$ of the slit $\gamma$.
\end{proposition}
To prove this proposition we need several lemmas, some of which will be used also in the next subsection. Again, fix \emph{any} parametrization $\Gamma\colon[0,T]\longrightarrow\uHc$, $T>0$, of the slit~$\gamma$.

For $s,t\in[0,T]$, $s\le t$, we define, see Figure~\ref{FG},
\begin{align*}
\varphi_{s,t}(z) &:=\big(g_{\gamma_t}^{-1}\circ g_{\gamma_s}\big)(z), \qquad z\in\uH,\\
\lambda(t) &:= g_{\gamma_t}^{-1}\big(\Gamma(t)\big)\in\R,\\
\mathcal{J}_{s,t} &:= g_{\gamma_t}^{-1}\big(\Gamma([s,t))\big)\subset\uH,\quad \bar{\mathcal J}_{s,t} := g_{\gamma_t}^{-1}\big(\Gamma([s,t])\big)=\mathcal J_{s,t}\cup\{\lambda(t)\},\\
\mathcal{C}_{s,t} &:= g_{\gamma_s}^{-1}\big(\Gamma([s,t])\big)\subset\Real.
\end{align*}
\begin{figure}[t]
\ignorespaces
\centering
\begin{tikzpicture}
    \coordinate (A) at (-2,0);
    \coordinate (B) at (2,0);

    \matrix[row sep=20pt, column sep=0pt] {
    \pgfmatrixnextcell
    \draw (A)--(B);
    \node [inner sep=0.75pt, minimum size=0pt, circle, label=left:$\R$] at (A) {};
    \foreach \x in {-2,-1.875,...,-1.125} \draw (\x,0)--(\x+0.06,-0.125);
    \draw[line width=1.5pt] (-1,0)--(1,0);
\node [inner sep=1pt, minimum size=0pt, draw, circle, fill] at (-1,0) {};
\node [inner sep=1pt, minimum size=0pt, draw, circle, fill] at (1,0) {};
    \node [inner sep=1pt, minimum size=0pt, draw, circle, fill, label=above:$\scriptstyle\lambda(s)$]
           at (0.25,0) {};
    \node [inner sep=0pt, minimum size=0pt, circle, label=below:$\mathcal{C}_{s,t}$] at (-0.25,0) {};
    \foreach \x in {1.05,1.175,...,2} \draw (\x,0)--(\x+0.06,-0.125);
    \node [inner sep=0pt, minimum size=0pt, circle, label=right:$\uH$] at ($(A)+(0,3)$) {};
    \node (HStoH) [inner sep=15pt, minimum size=0pt, circle] at (1,3) {};
    \node (HStoHT) [inner sep=0pt, minimum size=0pt, circle] at (1,-1) {};
    \pgfmatrixnextcell\pgfmatrixnextcell
    \draw (A)--(B);
    \node [inner sep=0.75pt, minimum size=0pt, circle, label=right:$\R$] at (B) {};
    \node [inner sep=0pt, minimum size=0pt, circle, label=right:${H_t:=\uH\setminus\gamma_t}$] at ($(B)+(0,3)$) {};
    \foreach \x in {-2,-1.875,...,2} \draw (\x,0)--(\x+0.06,-0.125);
    \draw (0,0) .. controls (0.5,0.25) and (-0.5,0.75) .. (0.25,1.25);
    \node [inner sep=1pt, minimum size=0pt, draw, circle, fill, label=right:$\vphantom{\int^1}{\scriptstyle\Gamma(t)}$]
           at (0.25,1.25) {};
    \draw[line width=1.5pt] (0.25,1.25) .. controls (1,1.6) and (-0.25,2.25) .. (0.25,2.5);
    \node [inner sep=1pt, minimum size=0pt, draw, circle, fill, label=right:$\scriptstyle\Gamma(s)$]
           at (0.25,2.5) {};
    \draw[dashed] (0.25,2.5) .. controls (0.75,2.75) and (0,3) .. (0.25,3.3);
    \node [inner sep=1pt, minimum size=0pt, draw, circle, label=right:$\scriptstyle\Gamma(0)$]
           at (0.25,3.3) {};
    \node [inner sep=5pt, minimum size=0pt, circle, label=left:$\gamma$] at (0.25,2.25) {};
    \node (HfromHS) [inner sep=15pt, minimum size=0pt, circle] at (-1,3) {};
    \node (HtoHT) [inner sep=0pt, minimum size=0pt, circle] at (-0.5,-1) {}; \\

    \pgfmatrixnextcell\pgfmatrixnextcell\pgfmatrixnextcell\\
    \pgfmatrixnextcell\pgfmatrixnextcell\pgfmatrixnextcell\\

    \pgfmatrixnextcell\pgfmatrixnextcell
    \draw (A)--(B);
    \node [inner sep=0.75pt, minimum size=0pt, circle, label=left:$\R$] at (A) {};
    \node [inner sep=7.5pt, minimum size=0pt, circle, label=below:$\uH$] at (0,0) {};
    \foreach \x in {-2,-1.875,...,2} \draw (\x,0)--(\x+0.06,-0.125);
    \draw[line width=1.5pt] (0,0) .. controls (0.5,0.25) and (-0.5,0.75) .. (0,1)
                                  .. controls (0.5,1.25) and (-0.5,1.5) .. (-0.25,1.75);
    \node [inner sep=1pt, minimum size=0pt, draw, circle, fill, label=above left:$\scriptstyle\lambda(t)$]
           at (0,0) {};
    \node [inner sep=1pt, minimum size=0pt, draw, circle, fill,
           label=above:$\scriptstyle\varphi_{s,t}(\lambda(s))$] at (-0.25,1.75) {};
    \node [inner sep=5pt, minimum size=0pt, circle, label=below right:$\bar{\mathcal{J}}_{s,t}$] at (0,1.75) {};
    \node (HTfromHS) [inner sep=10pt, minimum size=0pt, circle] at (-1.5,2) {};
    \node (HTfromH) [inner sep=10pt, minimum size=0pt, circle] at (1.5,2) {};
    \pgfmatrixnextcell \\ };

    \path[>=angle 45, line width=0.85pt] (HStoH)     edge[->,bend left=30]   node[above] {$g_{\gamma_s}$} (HfromHS)
                (HStoHT)    edge[->]                node[left] {$\varphi_{s,t}$} (HTfromHS)
                (HtoHT)     edge[->]                node[right] {$g_{\gamma_t}^{-1}$} (HTfromH);

\end{tikzpicture}
\ignorespacesafterend
\caption{Construction of $\varphi_{s,t}$, $\mathcal J_{s,t}$ and $\mathcal C_{s,t}$.}\label{FG}
\end{figure}
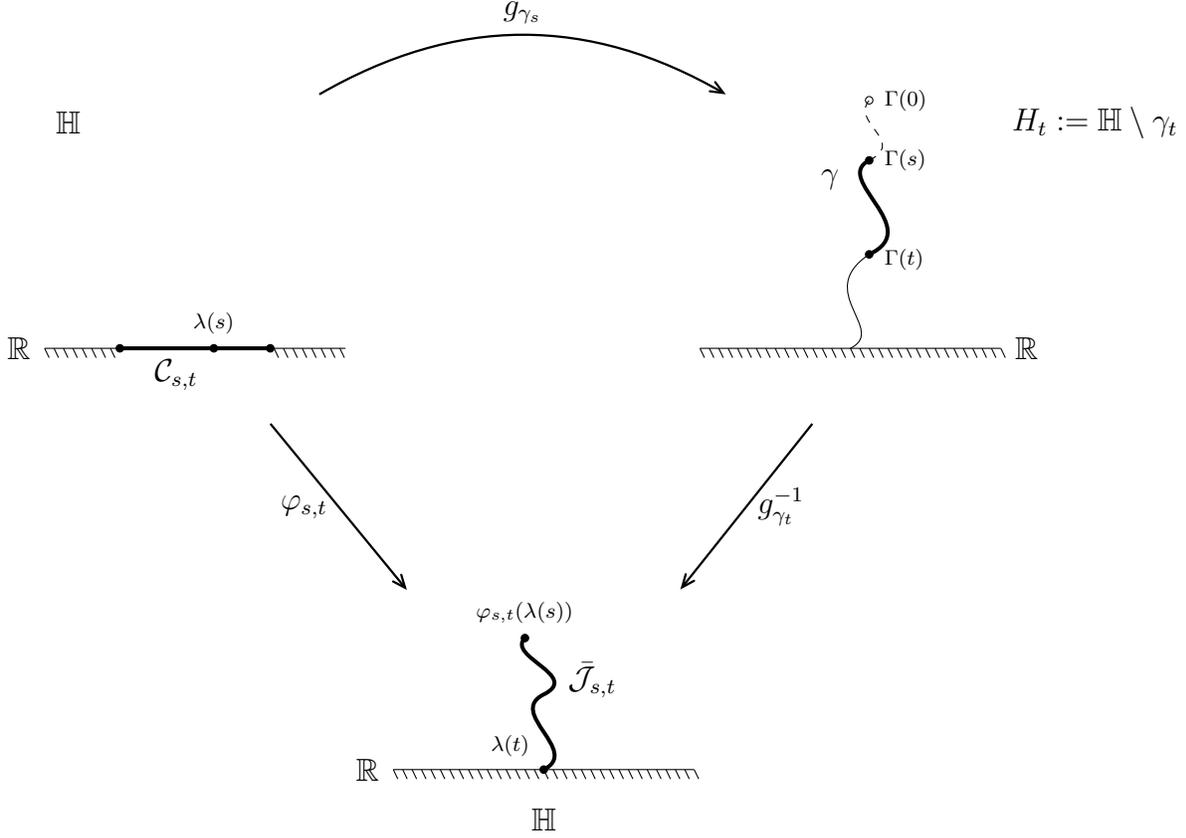
Since $\gamma_t\subset\gamma_s$, the functions $\varphi_{s,t}$ are well-defined conformal mappings of $\uH$ into itself. The set $g^{-1}_{\gamma_t}(\{\Gamma(t)\})$ consists, by Proposition~\ref{Pr_PropertiesSingleSlitMap}\,(ii), of a unique point, which makes $\lambda(t)$ be well and uniquely defined.
Moreover, for $s=t\in[0,T]$, clearly $\varphi_{s,t}$ is the identity map of~$\uH$, while if $0\le s<t\le T$ then, by Theorem~\ref{Th_PreOfSlitIsSlit},  $\varphi_{s,t}$ is a single-slit mapping. More precisely, using Propositions~\ref{Prop_RiemMapThm}, one easily obtains the following statement.
\begin{lemma}\label{LM:phi_st}
For any $s,t\in[0,T]$, $s<t$, the function $\varphi_{s,t}$ is a single-slit mapping with $\uH\setminus\varphi_{s,t}(\uH)=\mathcal J_{s,t}$ satisfying the hydrodynamic condition $\lim_{z\to\infty}\varphi_{s,t}(z)-z=0$. Moreover, $\varphi_{s,t}|_{\uH}$ extends to a conformal mapping $$\varphi_{s,t}^*\colon \ComplexE\setminus\mathcal C_{s,t}\Maponto\ComplexE\setminus\big(\bar{\mathcal J}_{s,t}\cup \bar{\mathcal J}_{s,t}^*\big),$$ where $\bar{\mathcal J}_{s,t}^*$ is the reflection of~$\bar{\mathcal J}_{s,t}$ w.r.t. the real axis. This extension has the Laurent expansion at~$\infty$ of the form
\begin{equation}\label{EQ_Laurent-phi_st}
\varphi_{s,t}^*(z)=z+\sum_{n=1}^{+\infty}c_{n}(s,t)z^{-n},
\end{equation}
with $c_{1}(s,t)=c_1(s)-c_1(t)<0$.
\end{lemma}
\begin{proof}
First of all, we may assume that $t<T$, because otherwise $\varphi_{s,t}=g_{\gamma_s}$ and hence the statement of the lemma would follow readily from  Proposition~\ref{Prop_RiemMapThm} applied with~$\Gamma([s,T))$ substituted for $\gamma$.

Now the fact that $\mathcal J_{s,t}$ is a slit in~$\uH$ follows directly from Theorem~\ref{Th_PreOfSlitIsSlit}. Since $g_{\gamma_t}^{-1}$ is a conformal mapping of $H_t:=\uH\setminus\gamma_t$ \emph{onto}~$\uH$, we have $$\varphi_{s,t}(\uH)=g_{\gamma_t}^{-1}(H_s)= g_{\gamma_t}^{-1}\big(H_t\setminus\Gamma([s,t))\big)= \uH\setminus g_{\gamma_t}^{-1}\big(\Gamma([s,t))\big)=\uH\setminus \mathcal J_{s,t}.$$
In particular, $\varphi_{s,t}$ is a single-slit mapping and extends to a continuous map from~$\uHc$ to~$\RS$.

By construction, $g_{\gamma_t}$ and $g_{\gamma_s}$ satisfy the hydrodynamic condition~\eqref{eq_HydroDynCond} and by Proposition~\ref{Prop_RiemMapThm}, these functions extend meromorphically to a neighbourhood of~$\infty$ having there the Laurent expansions of the form~\eqref{eq_FormLaurSerG}. It follows that $\varphi_{s,t}|_\uH$ also admits a meromorphic extension to a neighbourhood of~$\infty$, which has there the Laurent expansion of the form~\eqref{EQ_Laurent-phi_st} with $c_1(s,t)=c_1(s)-c_1(t)$. In particular, $\varphi_{s,t}$  satisfies the hydrodynamic condition and hence one can apply Proposition~\ref{Prop_RiemMapThm} with $\mathcal J_{s,t}$ substituted for~$\gamma$ to see that $c_1(s,t)<0$ and that $\varphi_{s,t}|_\uH$ extends by means of the Schwarz Reflection Principle to a conformal map~$\varphi^*_{s,t}$ of $\ComplexE\setminus \varphi_{s,t}^{-1}(\bar{\mathcal J}_{s,t})$ onto $\ComplexE\setminus\big(\bar{\mathcal J}_{s,t}\cup \bar{\mathcal J}_{s,t}^*\big)$. It remains to notice that the equality $g_{\gamma_s}(z)=g_{\gamma_t}\big(\varphi_{s,t}(z)\big)$ extends by continuity from~$\uH$ to its boundary and hence
$$%
\mathcal C_{s,t}=g_{\gamma_s}^{-1}\big(\Gamma([s,t])\big) = \varphi_{s,t}^{-1}\big(g_{\gamma_t}^{-1}\big(\Gamma([s,t])\big)\big)= \varphi_{s,t}^{-1}\big(\bar{\mathcal J}_{s,t}\big).
$$%
The proof is now complete.
\end{proof}

We will say that a sequence of Jordan arcs $\{\mathcal{C}_n\subset\Complex\}_{n\in\N}$  \emph{shrinks} to a point $p\in\C$ if $\mathcal C_{n+1}\subset\mathcal C_n$ and $\bigcap_{n\in\Natural} \mathcal C_n=\{p\}$. More generally, we will say that $\{\mathcal{C}_n\}_{n\in\N}$ \emph{tends} to a point~$p\in\C$ if  $d_n:=\sup\big\{|z-p|\colon z\in\mathcal C_n\big\}\to0$ as $n\to+\infty$.

\begin{lemma}\label{LM:ArcsShrinks}
For any fixed $t\in(0,T]$ the arc $\bar{\mathcal{J}}_{u,t}$ shrinks to the point $\lambda(t)$ and the segment $\mathcal{C}_{u,t}$ tends to the same point as $u\uparrow t$. Similarly,
for any fixed $s\in[0,T)$ the segment $\mathcal{C}_{s,u}$ shrinks to the point $\lambda(s)$ and the arc $\bar{\mathcal{J}}_{s,u}$ tends to the same point as $u\downarrow s$.
\end{lemma}
\begin{proof}
Fix $t\in(0,T]$. Then $\Gamma([u,t])$ shrinks to $\Gamma(t)$ as $u\uparrow t$. Since by Proposition~\ref{Pr_PropertiesSingleSlitMap}, $g_{\gamma_t}^{-1}(\{\Gamma(t)\})$ consists of a unique point, which we denote by $\lambda(t)$, it follows that $\bar{\mathcal J}_{u,t}$ shrinks to~$\lambda(t)$ as $u\uparrow t$. The same holds for the arcs $\mathcal I_{u,t}:=\bar{\mathcal J}_{u,t}\cup \bar{\mathcal J}_{u,t}^*$. Since the sets $\mathcal I_{u,t}$ are compact, it follows that $\diam \mathcal I_{u,t}\to0$ as $u\uparrow t$. Taking into account Lemma~\ref{LM:phi_st},  assertion~(ii) of Lemma~\ref{LM:varphi_extended}, applied with $\varphi:=\varphi^*_{s,t}$ and $w_0:=\lambda(t)$, implies that $\mathcal{C}_{u,t}$ tends to~$\lambda(t)$ as $u\uparrow t$.

Now fix $s\in[0,T)$. Then $\Gamma([s,u])$ shrinks to $\Gamma(s)$ as $u\downarrow s$. Hence, arguing essentially in the same way as above, we see that $\mathcal C_{s,u}$ shrinks to~$\lambda(s)$ and that $\bar{\mathcal J}_{u,t}$ tends to~$\lambda(s)$ as~$u\downarrow s$. The proof is finished.
\end{proof}

\begin{corollary}\label{CR_lambda-continuous}
The function $[0,T]\ni t\mapsto\lambda(t)$ is continuous.
\end{corollary}
\begin{proof}
Notice that, by construction, for any $s,t\in[0,T]$ such that $s<t$ we have $\lambda(s)\in\mathcal C_{s,t}$ and $\lambda(t)\in \bar{\mathcal J}_{s,t}$. Thus the continuity of~$t\mapsto\lambda(t)$ follows from the fact that by Lemma~\ref{LM:ArcsShrinks} both arcs, $\mathcal C_{s,t}$ and $\mathcal J_{s,t}$, tend to the same point as $t-s\to+0$ when one of the parameters, either $s$ or $t$, is fixed.
\end{proof}

\begin{lemma}\label{LM_c_1(t)}
The function $[0,T]\ni t\mapsto c_1(t)$ is continuous and strictly increasing.
\end{lemma}
\begin{proof}
The proof of the previous lemma shows that using assertion~(i) of Lemma~\ref{LM:varphi_extended} with $\varphi:=\varphi_{s,t}^*$ we may conclude that $c_1(s,u)\to0$ as $u\downarrow s$ for any fixed $s\in[0,T)$ and that $c_1(u,t)\to0$ as $u\uparrow t$ for any fixed~$t\in(0,T]$.
It remains to notice that by Lemma~\ref{LM:phi_st}, for any $s,t\in[0,T]$, $s<t$, we have $c_1(s)-c_1(t)=c_1(s,t)<0$.
\end{proof}

\begin{proof}[\textbf{Proof of Proposition~\ref{PR_standardParam}}]
Fix \emph{any} parametrization $\Gamma:[0,T]\to\uHc$ of the slit~$\gamma$. Then the proposition follows easily from Lemma~\ref{LM_c_1(t)} and the fact that~$c_1(T)=0$.
Indeed, consider another parametrization $\Gamma_0:[0,T_0]\to\uHc$ of the slit~$\gamma$. By definition it is standard if and only if $\Gamma_0\big(c_1(t)+T_0\big)=\Gamma(t)$ for all $t\in[0,T]$. Thus the unique standard parametrization is given by
$$
\Gamma_0:[0,T_0]\ni t\mapsto \big(\Gamma\circ\tau\big)(t-T_0),
$$
where $\tau$ is the inverse of $[0,T]\ni t\mapsto c_1(t)$ and $T_0:=-c_1(0)$.
\end{proof}

\subsection{Proof of Theorem~\protect{\ref{TH_Kufarev_etal}}}\label{SS_proof}
Let  $\gamma$ be a slit in~$\uH$ landing at a finite point on~$\Real$ and let ${\Gamma:[0,T]\to\uHc}$ be its unique standard parametrization, which exists due to Proposition~\ref{PR_standardParam}. To simplify the notation introduced in Section~\ref{SS_standParam} and to emphasize that now we work with the \emph{standard parametrization} of the slit, we will write $g_t$ instead of $g_{\gamma_t}$. The result of Kufarev \textsl{et al} (Theorem~\ref{TH_Kufarev_etal}) can be formulated in the following form.

\begin{theorem}\label{TH_Kuf-et-al_ODE}
There exists a unique continuous function $\lambda:[0,T]\to\Real$ such that for each $s\in[0,T)$ and each $z\in\uH$ the function
$t\in[s,T]\ni t\mapsto w_{z,s}(t):=\varphi_{s,t}(z)$ is the unique solution to the following initial value problem:
\begin{equation}\label{EQ_ini-chordal-L-ODE}
\frac{dw(t)}{dt}=\frac1{\lambda(t)-w(t)},\quad t\in[s,T];~\quad w(s)=z.
\end{equation}
\end{theorem}
See Remark~\ref{RM_ths_equiv} concerning the equivalence of Theorem~\ref{TH_Kufarev_etal} and Theorem~\ref{TH_Kuf-et-al_ODE}. In the proof of the latter we make use of the following two lemmas.
\begin{lemma}\label{LM_varphi_st}
Let $0\le s\le u\le t\le T$. The following statements hold:
\begin{itemize}
\item[(i)] $\varphi_{s,t}=\varphi_{u,t}\circ\varphi_{s,u}\vphantom{\displaystyle\int_0}$;
\item[(ii)] $\displaystyle \varphi_{s,t}(\zeta)=\zeta+\frac{1}{\pi}\int_{\mathcal{C}_{s,t}\!\!\!} \frac{\Im\{\varphi_{s,t}(\xi)\}}{\xi - \zeta}\;d\xi$\hskip1.5em for all \hskip.5em$\zeta\in\uH$;
\item[(iii)] $\displaystyle \samewd{c}{\varphi_{s,t}(\zeta)}{t - s} = \frac{1}{\pi}\int_{\mathcal{C}_{s,t}}\!\!\!\Im\big\lbrace\varphi_{s,t}(\xi)\big\rbrace\; d\xi$.
\end{itemize}
\end{lemma}
\begin{proof}
Assertion~(i) follows immediately from the definition of the functions $\varphi_{s,t}$. To prove~(ii) we recall that by Lemma~\ref{LM:phi_st}, $\varphi_{s,t}$ is a single-slit map with the hydrodynamic normalization. It follows that $f(\zeta):=\varphi_{s,t}(\zeta)-\zeta$, $\zeta\in\uH$, extends to a continuous map from~$\uHc$ into~$\Complex$. Further, by Proposition~\ref{Pr_PropertiesSingleSlitMap}, $\Im\big\lbrace f(\xi)\big\rbrace=\Im\big\lbrace\varphi_{s,t}(\xi)\big\rbrace=0$ for all~$\xi\in\Real\setminus\{\mathcal C_{s,t}\}$. Thus we may apply the Schwarz Integral Formula in the upper half-plane   (Proposition~\ref{pr_SchwarzIntegralForumula}) to $f$, which immediately yields~(ii).

Since we have chosen the standard parametrization of the slit~$\gamma$, by Lemma~\ref{LM:phi_st} we get $c_1(s,t)=s-t$ for any $s\ge0$ and any $t\ge s$. Therefore, substituting $\zeta:=iy$, $y>0$, in~(ii), multiplying both sides by~$-iy$, and passing to the limit as~$y\to+\infty$, one obtains~(iii). The proof is now complete.
\end{proof}
\begin{lemma}\label{LM:continuity-in-s}
For any $z\in\uH$ and any $s,t\in[0,T]$ with $s<t$,  $\varphi_{s,u}(z)\to \varphi_{s,t}(z)$ as $u\uparrow t$.
\end{lemma}
\begin{proof}
Denote $z_1=z_1(u):=\varphi_{s,u}(z)$ and $z_2:=\varphi_{s,t}(z)$. By Lemma~\ref{LM_varphi_st}\,(i), $z_2=\varphi_{u,t}(z_1)$. Note that by Lemma~\ref{LM:ArcsShrinks}, $\mathcal I_{u,t}:=\bar{\mathcal J}_{u,t}\cup \bar{\mathcal J}_{u,t}^*$ shrinks to the point $\lambda(t)\in\partial\uH$ as $u\uparrow t$ while $z_2\in\uH$ does not depend on~$u$. Hence, as in the proof of Lemma~\ref{LM:phi_st}, one can apply Lemma~\ref{LM:varphi_extended}\,(iv) with $\varphi:=\varphi_{u,t}^*$ to conclude that $z_1(u)\to z_2$ as $u\uparrow t$, which was to be shown.
\end{proof}

\begin{proof}[\textbf{Proof of Theorem~\ref{TH_Kuf-et-al_ODE}}]
Define $\lambda(t)$, as in Section~\ref{SS_standParam}, to be the unique preimage of~$\Gamma(t)$ under the map~$g_t$.
Then by Corollary~\ref{CR_lambda-continuous} the function~$\lambda$ is continuous on~$[0,T]$.
We are going to prove that $(\partial/\partial t) \varphi_{s,t}(z)$ exists and equals $1/\big(\lambda(t)-\varphi_{s,t}(z)\big)$ for any $z\in\uH$ and any $s,t\in[0,T]$ with $s\le t$. This will prove the existence of $\lambda$, while the uniqueness takes place because the function~$\lambda$ is determined uniquely by any solution to~\eqref{EQ_ini-chordal-L-ODE} with~$s=0$.

Let us fix $z\in\uH$ and $s,t\in[0,T]$ with $s\le t$.

\step1{we prove the existence and calculate the left derivative of $~t\mapsto\varphi_{s,t}(z)$.} So we assume $s<t$.
Take any $u\in[s,t)$. Then it follows from Lemma~\ref{LM_varphi_st} applied with $\zeta:=\varphi_{s,u}(z)$ that
\begin{equation*}
\frac{\varphi_{s,t}(z)-\varphi_{s,u}(z)}{t-u}=\frac{\varphi_{u,t}(\zeta)-\zeta}{t-u}= \raise1.5ex\hbox{$\displaystyle\int_{\mathcal C_{u,t}\!}\frac{\Im\{\varphi_{u,t}(\xi)\}}{\xi - \varphi_{s,u}(z)}\;d\xi$}%
\raise.3ex\hbox{$\Bigg/$}
\raise-.2ex\hbox{$\displaystyle\int^{\vphantom1}_{\mathcal{C}_{u,t}}\!\!\!\Im\big\lbrace\varphi_{u,t}(\xi)\big\rbrace\; d\xi$\,.}
\end{equation*}
Note that $\Im\big\lbrace\varphi_{u,t}(\xi)\big\rbrace\ge0$ for all $\xi\in\mathcal C_{u,t}$. By Lemma~\ref{LM:ArcsShrinks} the segment $\mathcal C_{u,t}$ tends to~$\lambda(t)$, while by  Lemma~\ref{LM:continuity-in-s}, $\varphi_{s,u}(z)\to\varphi_{s,t}(z)$ as  $u\uparrow t$. Hence using the Integral Mean Value Theorem, separately for the real and imaginary parts of $1/\big(\xi - \varphi_{s,u}(z)\big)$, we conclude that $\big(\varphi_{s,t}(z)-\varphi_{s,u}(z)\big)/(t-u)\longrightarrow 1/\big(\lambda(t) - \varphi_{s,t}(z)\big)$ as $u\uparrow t$.

\step2{now we prove the existence  and calculate the right derivative of $~t\mapsto\varphi_{s,t}(z)$.} We assume $t<T$.
Take any $u\in(t,T]$. Similarly to Step 1,
\begin{equation*}
\frac{\varphi_{s,u}(z)-\varphi_{s,t}(z)}{u-t}=\raise1.5ex\hbox{$\displaystyle\int_{\mathcal C_{t,u}\!}\frac{\Im\{\varphi_{t,u}(\xi)\}}{\xi - \varphi_{s,t}(z)}\;d\xi$}%
\raise.3ex\hbox{$\Bigg/$}
\raise-.2ex\hbox{$\displaystyle\int^{\vphantom1}_{\mathcal{C}_{t,u}}\!\!\!\Im\big\lbrace\varphi_{t,u}(\xi)\big\rbrace\; d\xi$\,.}
\end{equation*}
By Lemma~\ref{LM:ArcsShrinks} the segment $\mathcal C_{t,u}$ shrinks to~$\lambda(t)$ as $u\downarrow t$. Using again the Integral Mean Value Theorem, we see that $\big(\varphi_{s,u}(z)-\varphi_{s,t}(z)\big)/(u-t)\longrightarrow 1/\big(\lambda(t) - \varphi_{s,t}(z)\big)$ as $u\downarrow t.$

\step3{it remains to see that the solution to~\eqref{EQ_ini-chordal-L-ODE} is unique.} Notice that the vector field in the r.h.s. of~\eqref{EQ_ini-chordal-L-ODE}, $G(w)=1/(\lambda(t)-w)$, is Lipschitz continuous in~$w$ on every compact subset of~$\uH$, with the Lipschitz constant independent of~$t$. It remains to appeal to the standard uniqueness and existence theorem for initial value problems, see, \textsl{e.g.}, \cite[Chapter II, Theorem 1.1]{Hartman:OrdinaryDiffEquations}.
\end{proof}


\subsection{Some remarks}\label{SS_remarks}
First of all let us place a couple of remarks regarding Theorem~\ref{TH_Kuf-et-al_ODE}.

\begin{remark}\label{RM_ths_equiv}
Let us recall that under assumptions of Section~\ref{SS_proof}, $\varphi_{s,T}=g_s$ for all $s\in[0,T]$. Hence the chordal Loewner PDE
\begin{equation}\label{EQ_PDE_again}
\frac{\partial g_s(z)}{\partial s}=-\frac{g_s'(z)}{\lambda(s)-z}
\end{equation}
comes out of~\eqref{EQ_ini-chordal-L-ODE}  by appealing to the classical theorem about the dependence of solutions to an ODE on the initial data, see, \textsl{e.g.}, \cite[Chapter~V, Theorem~3.1]{Hartman:OrdinaryDiffEquations}. It is also quite easy to see that this PDE enforces the coefficient~$c_1(t)$ of $z^{-1}$ in the expansion of $g_t$ to be equal to $c_1(0)+t$. Therefore, Theorem~\ref{TH_Kufarev_etal} is, in fact, equivalent to Theorem~\ref{TH_Kuf-et-al_ODE}.
\end{remark}

\begin{remark}\label{RM_slit-non-slit}
Theorem~\ref{TH_Kuf-et-al_ODE} means, in particular, that the information about every slit~$\gamma$ is encoded in the corresponding unique real-valued function~$\lambda$. A natural question is whether the converse statement holds, \textsl{i.e.}, whether \emph{any} continuous real-valued function~$\lambda$ defined on a closed interval corresponds to a slit in~$\uH$. The answer is \emph{``no in general''}. A kind of converse theorem holds, but it only states that if $T>0$ and $\lambda:[0,T]\to\Real$ is a continuous (or, more generally, bounded measurable) function, then there exists a unique family $(g_t)_{t\in[0,T]}$ of univalent holomorphic self-maps of~$\uH$ with the hydrodynamic normalization such that for any $s\in[0,T)$ and $z\in\uH$ the function $w=w_{z,s}(t):=(g_t^{-1}\circ g_s)(t)$ solves the Cauchy problem~\eqref{EQ_ini-chordal-L-ODE}. However, $g_t$'s do \emph{not} need to be single-slit mappings. For further discussion and results in this direction see, \emph{e.g.}, \cite{Zhora:Non-slit} and references cited therein.
\end{remark}

\begin{unnumremark}
In the modern literature it seems to be a convention to rescale the independent variable~$t$ in the Loewner chordal equation, which leads to the extra coefficient 2 in the right-hand side:
\begin{equation}\label{EQ_chordal_ODE_with2}
\frac{dw(t)}{dt}=\frac{2}{\lambda(t)-w(t)}.
\end{equation}
This ``cosmetical'' change plays some role when comparing the chordal and radial Loewner ODEs, especially in connection with the question mention in Remark~\ref{RM_slit-non-slit}, see, \textsl{e.g.}, \cite{ProkhVas:singular_and_tangent}.
\end{unnumremark}

\begin{unnumremark}
Return again to the family $(g_t)_{t\in[0,T]}$ introduced in Section~\ref{SS_proof}. Consider the family of the inverse conformal mappings $(h_t)_{t\in[0,T]}$, $h_t:=g_t^{-1}:\uH\setminus\gamma_t\Maponto\uH$. Since $g_t(z)$ is of class $C^1$ jointly in~$z$ and~$t$, it follows from~\eqref{EQ_PDE_again} that $t\mapsto h_t(z)$ solves the chordal Loewner ODE. More precisely,
$$
\frac{\partial h_t(z)}{\partial t}=\frac{1}{\lambda(t)-h_t(z)}\quad\text{for all~$t\in[0,T]$ and all~$z\in H_t:=\uH\setminus\gamma_t$}.
$$
Although $(h_t)$ satisfies the same equation as $(\varphi_{s,t})$, the initial condition for~$(h_t)$ is given at the  \emph{right end-point}, $h_t|_{t=T}=\id_\uH$. Introducing the new parameter $\tau=T-t$ moves the initial condition to the left end-point~$\tau=0$ and brings the sign ``$-$'' to the right-hand side of the above equation. What is more important, this trick allows one to consider all $\tau\ge0$ and therefore to describe, by means of the chordal Loewner equation, cross-cuts in~$\uH$, \textsl{i.e.} Jordan arcs $\Gamma:[0,+\infty]\to\uHc$ joining, like a chord, two points on the boundary, $\Gamma(0)\in\Real$ and $\Gamma(+\infty)=\infty$, and otherwise lying in $\uH$. (This seems to be a plausible explanation for the word ``chordal'' in the name of the equation.)
\end{unnumremark}

The above two remarks bring the original chordal Loewner equation~\eqref{EQ_ODE-Loewner-chordal} to the form, which prevails in the recent literature:
$$
\frac{dw(t)}{dt}=\frac2{w(t)-\xi(t)},\quad  t\ge0,\qquad w(0)=z,
$$
where $\xi:[0,+\infty)\to\Real$ is a continuous function. As a function of the initial value~$z$, $w(t)$ maps its domain, \textsl{i.e.} the set of all~$z\in\uH$ for which the life-span~$T(z)$ of the solution to the above Cauchy problem is greater~$t$, conformally onto~$\uH$ and has the following expansion at~$\infty$,
$$
w=z+\frac{2t}{z}+\sum_{n=2}^{+\infty}a_n(t)z^{-n}.
$$

In the last lines of this survey paper,
it could be appropriate to mention that, up to our best knowledge, the chordal Loewner ordinary differential equation appeared for the first time as early as in 1946 (although without any further development) in Kufarev's paper~\cite{Kufarev1946}, the first paper approaching the problem indicated in Remark~\ref{RM_slit-non-slit}.


\begin{thebibliography}{99}
%
\bibitem{ABCDM:EvoLoewDiffEqn}
M. Abate, and F. Bracci, M.D. Contreras and S. D\'{i}az-Madrigal, The evolution
of Loewner's differential equations, Eur. Math. Soc. Newsl. No. 78 (2010),
31--38. MR2768999

\bibitem{Aleksandrov:ParConTheorUnivFunc}
I.A. Aleksandrov, {\it Parametric continuations in the theory of univalent
functions} (in Russian), Izdat. ``Nauka'', Moscow, 1976. MR0480952 (58 \#1099)

\bibitem{deBranges:ProofBierberbach}
L. de Branges, A proof of the Bieberbach conjecture, Acta Math. {\bf 154}
(1985), no.~1-2, 137--152. MR0772434 (86h:30026)

\bibitem{ColLoh:ClusterSets}
E.F. Collingwood\ and\ A. J. Lohwater, {\it The theory of cluster sets},
Cambridge Tracts in Mathematics and Mathematical Physics, No. 56 Cambridge
Univ. Press, Cambridge, 1966. MR0231999 (38 \#325)

\bibitem{Conway:FunComVa2}
J.B. Conway, {\it Functions of one complex variable. II}, Graduate Texts in
Mathematics, 159, Springer, New York, 1995. MR1344449 (96i:30001)

\bibitem{Duren:UnivFunc}
P.L. Duren, {\it Univalent functions}, Grundlehren  der Mathematischen
Wissenschaften, 259, Springer, New York, 1983. MR0708494 (85j:30034)

\bibitem{Goluzin:GeThFuCo}
G.M. Goluzin, {\it Geometric theory of functions of a complex variable},
Translations of Mathematical Monographs, Vol. 26 Amer. Math. Soc., Providence,
RI, 1969. MR0247039 (40 \#308)


\bibitem{GoryaBa:SemiGrpConfMap}
V.V. Gorya\u\i nov\ and\ I. Ba, Semigroup of conformal mappings of the upper half-plane into itself
with hydrodynamic normalization at infinity, Ukra\"\i n. Mat. Zh. {\bf 44} (1992), no. 10,
1320--1329; translation in Ukrainian Math. J. {\bf 44} (1992), no.~10,
1209--1217 (1993). MR1201130 (94b:30013)

\bibitem{Hartman:OrdinaryDiffEquations}P. Hartman, {\it Ordinary differential equations}, Wiley, New York, 1964. MR0171038 (30 \#1270)

\bibitem{Hayman:MultivalFunct}
W.K. Hayman, {\it Multivalent functions}, second edition, Cambridge Tracts in
Mathematics, 110, Cambridge Univ. Press, Cambridge, 1994. MR1310776 (96f:30003)

\bibitem{HockingYoung:Topology}
J.G. Hocking\ and\ G. S. Young, {\it Topology}, Addison-Wesley Publishing Co.,
Inc., Reading, MA, 1961. MR0125557 (23 \#A2857)

\bibitem{Zhora:Non-slit} G. Ivanov, D. Prokhorov\ and\ A. Vasil'ev, Non-slit and singular solutions to the L\"owner equation, Bull. Sci. Math. {\bf 136} (2012), no.~3, 328--341. MR2914952

\bibitem{Kufarev:FamiliesOfAnalFunctions}
P.P. Kufarev,  On one-parameter families of analytic functions (in
Russian. English summary), Rec.~Math. [Mat. Sbornik] N.S. \textbf{13 (55)}
(1943), 87--118. MR0013800 (7,201g)

\bibitem{Kufarev1946} P.P. Kufarev, On integrals of simplest differential equation with moving
 pole singularity in the right-hand side, Uchen. Zap. Tomsk. Gos. Univ. (1946), no.\,1,
 35--48.

\bibitem{KufarevEtAl:FuncUnvHalfPl}
P.P. Kufarev, V.V. Sobolev\ and\ L.V. Spory\v seva, A certain method of
investigation of extremal problems for functions that are univalent in the
half-plane (in Russian), Trudy Tomsk. Gos. Univ. Ser. Meh.-Mat. {\bf 200} (1968), 142--164.
MR0257336 (41 \#1987)

\bibitem{Lawler:ConforInvarProces}
G.F. Lawler, {\it Conformally invariant processes in the plane}, Mathematical
Surveys and Monographs, 114, Amer. Math. Soc., Providence, RI, 2005. MR2129588
(2006i:60003)

\bibitem{LawSchramWend:ValBrowIntExp}
G.F. Lawler, O. Schramm\ and\ W. Werner, Values of Brownian intersection
exponents. I. Half-plane exponents, Acta Math. {\bf 187} (2001), no.~2,
237--273. MR1879850 (2002m:60159a)

\bibitem{Loewner:schlichte}
K. L\"owner, Untersuchungen \"uber schlichte konforme Abbildungen des Einheitskreises. I,
Math. Ann. {\bf 89} (1923), no.~1-2, 103--121. MR1512136

\bibitem{Pommerenke:SubordinationChains}Ch. Pommerenke, \"{U}ber die subordination analytischer funktionen, J. Reine Angew Math. \textbf{218} (1965), 159--173. MR0180669 (31 \#4900)

\bibitem{Pommerenke:UnivalFunctions}
Ch. Pommerenke, {\it Univalent functions. With a chapter on quadratic
differentials by Gerd Jensen.}, Vandenhoeck \&\ Ruprecht, G\"ottingen, 1975.
MR0507768 (58 \#22526)

\bibitem{Pommerenke:BoundBehaConfMaps}
Ch. Pommerenke,  {\it Boundary behaviour of conformal maps}, Grundlehren der
Mathematischen Wissenschaften, 299, Springer, Berlin, 1992. MR1217706
(95b:30008)

\bibitem{ProkhVas:singular_and_tangent} D.\,Prokhorov and A.\,Vasil'ev, Singular and tangent slit solutions to the L\"owner equation, in {\it Analysis and mathematical physics}, 455--463, Trends Math, Birkh\"auser, Basel. MR2724626 (2012a:30015)

\bibitem{Schramm:ScLimUnSpTre}
O. Schramm,  Scaling limits of loop-erased random walks and uniform spanning
trees, Israel J. Math. {\bf 118} (2000), 221--288. MR1776084 (2001m:60227)

\bibitem{Tsuji:PotenTheoModFuncTheo}
M. Tsuji, {\it Potential theory in modern function theory}, Chelsea, New York, 1975. MR0414898 (54 \#2990)


\end{thebibliography}
\end{document}